\documentclass[a4paper,11pt,reqno]{amsart}
\usepackage{amsthm}
\usepackage{amsmath}
\usepackage{enumerate}
\usepackage{amsfonts}
\usepackage{amssymb}
\usepackage{fullpage}
\usepackage{amsmath,amscd}
\usepackage{stmaryrd}
\usepackage{graphicx}
\usepackage{array}
\usepackage{amsmath,amssymb,amsfonts,dsfont}
\usepackage[utf8]{inputenc}
\usepackage[english]{babel}
\usepackage[T1]{fontenc}
\usepackage{graphicx}
\usepackage{float}
\usepackage{pdfpages}
\usepackage[a4paper, margin = 3cm, bottom = 3cm]{geometry}
\usepackage{ifpdf}
\usepackage{marginnote}
\usepackage{hyperref}	
\hypersetup{pdfborder=0 0 0,
	    colorlinks=true,
	    citecolor=black,
	    linkcolor=blue,
	    urlcolor=red,
	    pdfauthor={Guillaume Tahar}
	   }
\newtheorem{thm}{Theorem}[section]

\newtheorem{prop}[thm]{Proposition}
\newtheorem{lem}[thm]{Lemma}
\theoremstyle{definition}
\newtheorem{defn}[thm]{Definition}
\theoremstyle{remark}
\newtheorem{rem}[thm]{Remark}
\theoremstyle{definition}

\theoremstyle{definition}

\theoremstyle{definition}

\numberwithin{equation}{section}
\title{Horizon saddle connections, quasi-Hopf surfaces and Veech groups of dilation surfaces}

\author{Guillaume Tahar}
\address[Guillaume Tahar]{Faculty of Mathematics and Computer Science, Weizmann Institute of Science,
Rehovot, 7610001, Israel}
\email{tahar.guillaume@weizmann.ac.il}

\date{May 12, 2020}

\begin{document}
\begin{abstract}
Dilation surfaces are generalizations of translation surfaces where the geometric structure is modelled on the complex plane up to affine maps whose linear part is real. They are the geometric framework to study suspensions of affine interval exchange maps. However, though the $SL(2,\mathbb{R})$-action is ergodic in connected components of strata of translation surfaces, there may be mutually disjoint $SL(2,\mathbb{R})$-invariant open subsets in components of dilation surfaces.  A first distinction is between triangulable and non-triangulable dilation surfaces. For non-triangulable surfaces, the action of $SL(2,\mathbb{R})$ is somewhat trivial so the study can be focused on the space of triangulable dilation surfaces.\newline
In this paper, we introduce the notion of horizon saddle connections in order to refine the distinction between triangulable and non-triangulable dilation surfaces. We also introduce the family of quasi-Hopf surfaces that can be triangulable but display the same trivial behavior as non-triangulable surfaces. We prove that existence of horizon saddle connections drastically restricts the Veech group a dilation surface can have.\newline
\end{abstract}
\maketitle
\setcounter{tocdepth}{1}
\tableofcontents

\section{Introduction}
A dilation structure on a surface is a geometric structure modelled on the complex plane with structural group the set of maps $x \mapsto ax+b$ with $a \in \mathbb{R}_{+}^{*}$ and $b \in \mathbb{C}$. In most cases, a dilation structure is defined on the complement of a finite set of singularities, just like translation structures, see \cite{Zo}.\newline
Dilation surfaces have many common features with translation surfaces: notions of straight lines, angles and direction are well-defined (they are defined on $\mathbb{C}$ and transported by pullback to the dilation surface). In particular, for every angle $\theta \in S^{1}$, we can define a \textit{directional foliation}. whose leaves are straight lines in direction $\theta$ in every chart. These leaves are the \textit{trajectories} of the dynamics defined by the directional foliation. Just like translation surfaces are suspensions of interval exchange maps. Dilation surfaces are suspensions of affine interval exchange maps. These foliations have been studied in \cite{BFG}. Just like moduli spaces of translation surfaces, moduli spaces of dilation surfaces decompose into strata that are analytic orbifolds on which there is an action of $SL(2,\mathbb{R})$ that encompasses a renormalization process.\newline
In strata of translation surfaces, most interesting dynamic properties depend on the closure of the $SL(2,\mathbb{R})$-orbit. Besides, the orbit of a generic translation surface is dense in its connected component of the stratum. We look for an analogous framework for dilation surfaces.\newline
However, there are no $SL(2,\mathbb{R})$-invariant dense open sets in strata of dilation surfaces. An idea exposed in \cite{DFG} is that we should focus on nice subsets of strata. A first criterion is to consider dilation surfaces that are triangulable in a geometric sense. We require that the edges of the triangulation are saddle connections, that is geodesic segments between singularities (and without singularities in their interior). These triangulable dilation surfaces form a $SL(2,\mathbb{R})$-invariant open set of the moduli space.\newline
In this paper, we generalize this criterion by introducing the notion of \textit{horizon saddle connections}. If a dilation surface displays such objects, then the action of $SL(2,\mathbb{R})$ on it simplifies drastically.\newline

\begin{defn}
In a dilation surface, a $k$-horizon saddle connection $\gamma$ is a saddle connection such that no trajectory crosses $\gamma$ strictly more than $k$ times. 
\end{defn}

In particular, non-triangulable dilation surfaces admit $1$-horizon saddle connections, see Proposition 3.1 for details.\newline

Inclusion of the moduli space $\mathcal{T}$ of triangulable dilation surfaces in the moduli space $\mathcal{D}$ of dilation surfaces can be generalized into an infinite sequence of inclusions of $SL(2,\mathbb{R})$-invariant open sets. The example of quasi-Hopf surfaces (see Subsection 3.2 for details) shows that there are triangulable surfaces with some horizon saddle connections.

\begin{thm}
In the moduli space of dilation surfaces $\mathcal{D}$, for any $k \geq 1$ the set $\mathcal{H}_{k}$ of surfaces without $k$-horizon saddle connections is a $SL(2,\mathbb{R})$-invariant open set. We have
$$\dots \subset \mathcal{H}_{k} \subset \dots \subset \mathcal{H}_{1} \subset \mathcal{T} \subset \mathcal{D}$$
\end{thm}

Theorem 1.2 is proved in Section 3.\newline

Horizon saddle connections constrain the action of $SL(2,\mathbb{R})$. In the following, the Veech group $V(X)$ of a dilation surface $X$ is the stabilizer of the action of $SL(2,\mathbb{R})$. There is a remarkable similarity with the three types of Veech groups for translation surfaces with poles: continuous, cyclic parabolic or finite, see \cite{Ta,Ta2} for details.

\begin{thm}
In a dilation surface $X$ such that there are at least three distinct directions of horizon saddle connections, then the Veech group of $X$ is finite.
\end{thm}

\begin{thm}
In a dilation surface $X$ such that there are exactly two distinct directions of horizon saddle connections, then there are two cases:\newline
(i) $V(X)$ is finite (of order $1$, $2$ or $4$);\newline
(ii) the Veech group of $X$ is conjugated to $\left\{ \begin{pmatrix} a^{k}&0 \\ 0 & a^{-k} \end{pmatrix} \mid k \in \mathbb{Z} \right\} $ with $a \in  \mathbb{R}^{*}_{+}$ or its product with the subgroup generated by $\begin{pmatrix} 0&1 \\ -1 & 0 \end{pmatrix}$.\newline
In the latter case, the surface is a rational quasi-Hopf surface.
\end{thm}

\begin{thm}
In a dilation surface $X$ such that there is exactly one direction of horizon saddle connections, then there are three cases:\newline
(i) Hopf surfaces: the Veech group of $X$ is conjugated to $\left\{\begin{pmatrix} a & b \\ 0 & a^{-1} \end{pmatrix} \mid a \in \mathbb{R}^{*}, b \in \mathbb{R} \right\}$;\newline
(ii) the Veech group of $X$ is cyclic parabolic: conjugated to $\left\{ \begin{pmatrix} 1&k \\ 0 & 1 \end{pmatrix} \mid k \in \mathbb{Z} \right\}$ or its product with $\left\{ \pm Id \right\}$.\newline
(iii) the Veech group of $X$ is trivial or $\left\{ \pm Id \right\}$.
\end{thm}

Theorems  1.3 to 1.5 are proved in Section 4.\newline

The structure of the paper is the following: \newline
- In Section 2, we recall the background about dilation surfaces: moduli space, linear holonomy, triangulations.\newline
- In Section 3, we introduce the notion of horizon saddle connections and prove the finiteness lemma. We also discuss some examples: dilation tori, surfaces with chambers, quasi-Hopf surfaces, dilation surfaces with finitely many saddle connections.\newline
- In Section 4, we present some preliminary results about Veech groups and prove the main theorems about restriction of the Veech group in presence of horizon saddle connections.

\section{Dilation surfaces}

\subsection{Dilation structure}

A dilation structure on a Riemann surface $X$ is a kind of affine structure. We follow the definitions of \cite{DFG}.

\begin{defn}
A dilation surface is a Riemann surface $X$ with a finite set $\Lambda \subset X$ of singularities and an atlas of charts on $X \setminus \Lambda$ with values in $\mathbb{C}$ and such that:\newline
(i) transition maps are of the form $x \mapsto ax+b$ with $a \in \mathbb{R}_{+}^{*}$.\newline
(ii) the affine structure extends around every element of $\Lambda$ to a "conical singularity" characterized by its topological index and its dilation ratio (see Subsection 2.2).
\end{defn}

One difficulty in the study of dilation surfaces is that there is no notion of distance since two segment with different lengths and the same direction are equivalent up to affine maps. However, the ratio of lengths of two saddle connections sometimes makes sense.\newline
Let $\alpha$ and $\beta$ be two saddle connections of a dilation surface $X$ that intersect (possibly at the ends of the segment) each other. Then, we consider a chart covering a disk $D$ (possibly with a slit) centered on the intersection point. In $D$, the length of the two intersecting branches is well-defined. The portion these branches represent in their saddle connection is also well-defined. Therefore, the \textit{local length ratio} of $\alpha$ and $\beta$ is well-defined. It only depends on the choice of the intersection point and possibly the order in the couple $(\alpha,\beta)$ if the intersection point is a singularity with a nontrivial dilation ratio. This construction will help in some crucial results about Veech groups.

\subsection{Linear holonomy}

In a dilation surface $X$, for every closed path $\gamma$ of $X \setminus \Lambda$, we can cover $\gamma$ with charts of the atlas. The transition map between the first chart and the last chart is an affine map. Its linear part is well-defined up to conjugacy. This number is obviously a topological invariant. Therefore, we have a representation of the pointed fundamental group of $X \setminus \Lambda$ into $\mathbb{R}_{+}^{*}$. Since the latter is Abelian, the representation factorizes through $H_{1}(X \setminus \Lambda, \mathbb{Z})$ (Hurewicz theorem). Thus, we get a group morphism $\rho: H_{1}(X \setminus \Lambda, \mathbb{Z}) \longrightarrow \mathbb{R}_{+}^{*}$ we denote by \textit{linear holonomy} of loops.

The local geometry of a conical singularity is determined by two topological numbers associated to a simple loop $\gamma$ around it:\newline
(i) the linear holonomy $\rho(\gamma)$;\newline
(ii) the topological index $i(\gamma)$.\newline

The neighborhood of such a singularity is constructed starting from an infinite cone of angle $i(\gamma)2\pi$ and a ray starting from the origin of the cone. Then, we identify the right part of the ray with the left part of the ray with a homothety ratio of $\rho(\gamma)$.\newline

The conical singularities satisfy a Gauss-Bonnet formula. It is remarkable that the dilation ratio appears as something like an "imaginary curvature". The following formula has been proved as Proposition 1 in \cite{G}.

\begin{prop}
In a surface of dilation of genus $g \geq 1$ with conical singularities $s_{1},\dots,s_{n}$ of angle $2k_{i}\pi$ and dilation ratio $\lambda_{i}$, we have:\newline
(i) $\sum_{i=1}^{n} (k_{i}-1) = 2g-2$;\newline
(ii) $\sum_{i=1}^{n} ln(\lambda_{i}) = 0$.
\end{prop}

In particular, there is no dilation surface in genus zero. It would be necessary to introduce a notion of pole among the singularities of the affine structure.

\subsection{Moduli space and strata}

For any $g,n \geq 1$, we consider the set $\mathcal{X}$ of dilation structures on a given compact topological surface of genus $g$ with $n$ marked points. We define the moduli space $\mathcal{D}_{g,n}$ of dilation surfaces of genus $g$ with $n$ singularities as the quotient of $\mathcal{X}$ by the group of orientation-preserving diffeomorphisms. This space is an analytic orbifold of real dimension $6(g-1)+3n$, see \cite{V} for details.\newline

For any sequence of integers $a=(a_{1},\dots,a_{n})$ of integers such that $\sum_{i=1}^{n} a_{i} = 2g-2$ and any sequence of positive real numbers $\lambda=(\lambda_{1},\dots,\lambda_{n})$ such that $\prod_{i=1}^{n} \lambda_{i} = 1$, there is a (possibly empty) stratum $\mathcal{D}_{g,n}(a,\lambda)$ of $\mathcal{D}_{g,n}$. These strata are analytic orbifolds of real dimension $6(g-1)+2n+1$.\newline

$SL(2,\mathbb{R})$ acts on the moduli space of dilation surface by composition with the coordinates maps with values in $\mathbb{C}$. The action preserves the linear holonomy $\rho$ so it also preserves strata. The Veech group $V(X)$ of a dilation surface $X$ is the stabilizer of this group action. It is a subgroup of $SL(2,\mathbb{R})$.

\subsection{Cylinders}

For any closed geodesic, the first return of a small segment orthogonal to the geodesic is a map of the form $x \mapsto \lambda x$ with $\lambda \in \mathbb{R}_{+}^{*}$. We say that the closed geodesic is flat if $\lambda = 1$. Otherwise, it is hyperbolic.\newline

Flat closed geodesics belong to families that describe flat cylinders (just like in translation surfaces). Hyperbolic closed geodesics also describe cylinders. We refer to them as \textit{affine cylinders}. They are obtained from an angular portion of an annulus (or a cyclic cover of an annulus) by gluing the two sides on each other, see Figure 1. The isomorphism class of an affine cylinder is completely determined by two numbers:\newline
(i) The affine factor $\lambda$ (dilation ratio along hyperbolic closed geodesics).\newline
(ii) The angle $\theta$ (determined by the angular portion of the annulus considered).\newline

\begin{figure}
\includegraphics[scale=0.3]{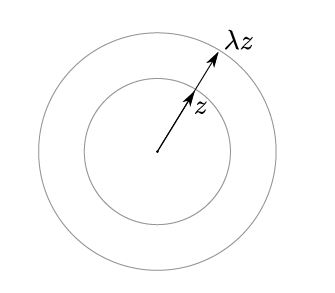}
\caption{Hyperbolic cylinder of angle $2\pi$ and dilation ratio $\lambda$. This surface formed by a unique hyperbolic cylinder is a \textit{Hopf torus}.}
\end{figure}

In translation surfaces, the horocyclic flow and its conjugates (action of unipotent elements of $SL(2,\mathbb{R})$ preserving the direction of the closed geodesics of the cylinder) modify the twist of a flat cylinder. There is a similar phenomenon for affine cylinders.\newline

We consider an affine cylinder of dilation factor $\lambda$ whose boundary is formed by a saddle connection with a marked point in direction $\alpha$ and another saddle connection with a marked point in direction $\beta$. Then we consider the conjugate of the Teichmüller flow $A_{t}$ that contracts direction $\alpha$ with a factor $e^{-t}$ and expands direction $\beta$ with a factor $e^{t}$. The flow of this subgroup on this surface is periodic and the affine cylinder is invariant by any element $A_{t}$ of the flow such that $t=\frac{ln(\lambda)}{2}$. This will be useful in Subection 3.3 when we will study quasi-Hopf surfaces.

\subsection{Triangulable dilation surfaces}

In the framework of dilation surfaces, a (geometric) triangulation is a topological triangulation where the edges are saddle connections and where every conical singularity is a vertex. Having a triangulation of dilation surface provides a nice parametrization of its neighborhood in the moduli space (by deforming the triangles). However, not every dilation surface admits a triangulation. Indeed, affine cylinders of angle at least $\pi$ are not triangulable.\newline

A geodesic trajectory entering in an affine cylinder cannot intersect a hyperbolic geodesic that shares the same direction. Let $I \subset  S^{1}$ be the closure of the interval of the directions of closed geodesics of the affine cylinder. If a trajectory $\gamma$ whose direction do not belong to $I$ enters the cylinder, then it will cross every closed geodesic and finally leave by the opposite boundary. On the contrary, if its direction $\theta$ belongs to $I$, then $\gamma$ will accumulate on the hyperbolic geodesic of direction $\theta$ the closest to the entry point (in the foliation of the cylinder by closed geodesics).\newline
In particular, if the angle of the affine cylinder is at least $\pi$, then there is no saddle connection joining the two boundaries of the cylinder. If there were such a segment, it would belong to a direction that is a direction of a hyperbolic geodesic of the cylinder and it would be forced to cross it. Thus, there is no geometric triangulation of the affine cylinder.\newline

Veech proved in some unpublished course notes the converse result (see \cite{V2}).

\begin{thm}[Veech]
A dilation surface $X$ admits a geometric triangulation if and only if there is no affine cylinders of angle at least $\pi$ in $X$.
\end{thm}

The idea of the proof is that without affine cylinders of angle at least $\pi$, every affine immersion of an open disk extends to an immersion of its closure. This geometric lemma allows to build a Delaunay triangulation whose dual is a geometric triangulation. The proof clearly generalizes to the case of dilation surfaces with geodesic boundary. Besides, a different proof is given in \cite{Ta3} by proving that any dilation surface with boundary that is not a triangle or an affine cylinder of angle at least $\pi$ has an interior saddle connection.\newline

Even if the surface is not triangulable, a maximal system of non-intersecting saddle connections cuts out the surface into a union of affine cylinders of angle at least $\pi$ (they are clearly disjoint from each other) and a triangulable locus (without affine cylinders of angle at least $\pi$, the construction of a triangulation in the surface with boundary works). It will be used as a substitute of triangulation.

\section{Horizon saddle connections}

A horizon saddle connection is a segment such that for some number $k$, there is no geodesic trajectory crossing it strictly more than $k$ times. In particular, if $k<k'$, any $k$-horizon saddle connection is a $k'$-saddle connection. There is no such horizon saddle connections in translation surfaces. Indeed, the foliation in the generic direction is minimal so every saddle connection is crossed infinitely many times by some trajectory. Therefore, absence of horizon saddle connections indicates that we are not too far from the case of translation surfaces.\newline

The following proposition is the starting point of the generalization of our triangulability condition.

\begin{prop}
The saddle connections that belong to the boundary of an affine cylinder of angle at least $\pi$ are $1$-horizon saddle connections.
\end{prop}

\begin{proof}
A trajectory that enters in such an affine cylinder cannot leave it (they accumulate on an hyperbolic geodesic of the cylinder, see Subsection 2.1 for details). Therefore, no saddle connection of the boundary can be crossed two times.
\end{proof}

The key property of horizon saddle connections is that for any $k$, there are finitely many $k$-horizon saddle connections. This property drastically rigidifies the action of $SL(2,\mathbb{R})$. In a similar way, in translation surfaces with poles, most the geometry of a surface is encompassed in a polygon called the core of the surface. The fact that there are finitely many saddle connections in the boundary of the core makes more rigid the action of $SL(2,\mathbb{R})$ on strata of such surfaces, see \cite{Ta2} for details.

\begin{lem}
In a dilation surface, for any $k \geq 1$, there is at most a finite number of $k$-horizon saddle connections.
\end{lem}

\begin{proof}
We choose a maximal geodesic arc system of $X$ (that is a decomposition into disjoint triangles and affine cylinders of angle at least $\pi$). For any $k$, there is a finite number of free homotopy classes of loops that cross each edge of the system at most $k$ times. There is a unique geodesic representative in every homotopy class and it minimizes the geometric intersection number. Therefore, there cannot be infinitely many distinct $k$-horizon saddle connections.
\end{proof}

\begin{rem}
We actually do not know if there is dilation surface with $k$-horizon saddle connections for arbitrary high $k$. We expect that there is a topological bound on the maximal number of the maximal crossing number of a horizon saddle connection.
\end{rem}

Just like triangulable dilation surfaces define a $SL(2,\mathbb{R})$-invariant open set, absence of $k$-horizon saddle connections defines for every $k$ an invariant open set. 

\begin{proof}[Proof of Theorem 1.2]
Let $X$ be a surface without $k$-horizon saddle connections. We choose a geometric triangulation of $X$ and a neighborhood of $X$ in the moduli space such that the same geodesic triangulation holds in the neighborhood. For any topological arc that crosses the edges more than $k$ times (there is a finite number of such arcs), its geodesic representative in $X$ is formed by several saddle connections each of which is crossed by trajectories with an arbitrary high number of intersections. Such a trajectory still exists in a neighborhood of $X$. Thus, for each arc, there is neighborhood where this arc cannot be a horizon saddle connection (at least a compact part of the trajectory that crosses more than $k$ times persists). Since there is a finite number of such arcs, there is a neighborhood of $X$ where there is no horizon saddle connections.
\end{proof}

\subsection{Genus one}

Even in the simplified situation of dilation surfaces of genus one, triangulated surfaces and surfaces free from horizon saddle connections define distinct open subsets of the strata.\newline

In \cite{G2}, Ghazouani proved that any dilation torus with $n$ singularities can be decomposed into at most $n$ flat and affine cylinders (see Proposition 9). The sketch of the proof is the following: Gauss-Bonnet formula implies that every singularity has a conical angle of $2\pi$. For a given transverse curve, the homeomorphism of the circle induced by the directional flow changes continuously and monotonically with the direction. Therefore, in some directions, the rotation number is rational and there is at least one close geodesic. This closed geodesic belongs to a cylinder. This cylinder is bounded by a chain of saddle connections joining marked points. Since they have a conical angle of $2\pi$, the other side of the chain of saddle connections also belongs to a cylinder.\newline

A dilation surface belongs to the locus of triangulated surfaces $\mathcal{TD}_{1,n}$ in $\mathcal{D}_{1,n}$ if and only if every affine cylinder of the decomposition has an angle strictly smaller than $\pi$.\newline

In the following proposition, we introduce a significantly stronger criterium.

\begin{prop}
For a dilation surface $X$ of genus one that admits a decomposition into $c$ cylinders. Let $A_{1},\dots,A_{c} \subset S^{1}$ be the closures of the sets of directions of geodesics of the cylinders of the decomposition. Exactly one of the following two statements holds:\newline
(i) $\bigcup\limits_{i=1}^{c} A_{i} = S^{1}$ and every boundary saddle connection of the decomposition is a $1$-horizon saddle connection. There are no other horizon saddle connections in the surface.\newline
(ii) $\bigcup\limits_{i=1}^{c} A_{i} \neq S^{1}$ and every saddle connection is crossed infinitely many times by some trajectory.
\end{prop}

\begin{proof}
In the first case, there is a closed geodesic in any direction. Therefore, for any trajectory $\alpha$, there will be a closed geodesic $\gamma$ (or equivalently a chain of saddle connections separating two cylinders of the decomposition) in the same direction $\alpha$ cannot cross. For any saddle connection $\beta$ that is the boundary of two cylinders, it is clear that if $\alpha$ crosses $\beta$ two times, then it should cross $\gamma$ at least one time. Conversely, any saddle connection that does not separate cylinders is crossed infinitely many times by the closed geodesics of the cylinder.\newline

In the second case, we consider a closed geodesic $\gamma$ whose direction is in $\bigcup\limits_{i=1}^{c} A_{i}$ and a direction $\theta$ in the complement of $\bigcup\limits_{i=1}^{c} A_{i}$ that is not the direction of some saddle connection either (since the complement is any open set, it is always realizable). Then we consider the first return map $\phi$ on $\gamma$ defined by the directional flow in direction $\theta$ (every trajectory in such a direction eventually leaves every cylinder whose geodesics have directions $\bigcup\limits_{i=1}^{c} A_{i}$). Map $\phi$ is a homeomorphism of the circle. There are two cases. If the rotation number of $\phi$ is irrational, then the directional foliation in direction $\theta$ has minimal leaves that are dense in the whole surface and cross every saddle connection infinitely many times. If the rotation number is rational, then there is a closed geodesic in direction $\theta$, then it defines a cylinder whose closed geodesics intersect every cylinder of the first decomposition. This cylinder can be completed to form another decomposition into cylinders. Closed geodesics of a given decomposition represent the same free homotopy class. The intersection number of the loops of the two decompositions is not trivial. Therefore, every saddle connection in the boundary of a cylinder of the first decomposition is crossed infinitely many times by some trajectory.
\end{proof}

In any stratum $\mathcal{D}_{1,n}(\lambda)$, we define $\mathcal{H}_{1,n}(\lambda)$ to be the $SL(2,\mathbb{R})$-invariant open set (see Theorem 1.2) formed by surface without horizon saddle connections. We have the following strict inclusions of invariant open sets:\newline
$$ \mathcal{H}_{1,n}(\lambda) \subsetneq \mathcal{DT}_{1,n}(\lambda) \subsetneq \mathcal{D}_{1,n}(\lambda)$$

We proved a dichotomy among dilation tori between those with horizon saddle connections (where one cylinder decomposition covers every direction) and horizon-free dilation tori. In the first case, there is only one cylinder decomposition because any other decomposition would contain a closed geodesic that would cross every cylinder of the first decomposition. Horizon saddle connections make this situation impossible. Therefore, the shape of the cylinders of the unique cylinder decomposition provides global coordinates for this locus in the moduli space. $SL(2,\mathbb{R})$ acts separately on each cylinder of the decomposition. We should not expect any interesting recurrence behaviour.

\subsection{Surfaces with chambers}

A natural question in the study of moduli spaces of geometric structures is about the connected components. In \cite{DFG}, the authors study the space $\mathcal{DT}_{2,1}$ of triangulable dilation surfaces of genus two with only one conical singularity. This space fail to be connected and there is an exceptional connected component formed by surfaces split into two chambers separated by a closed saddle connection. This cannot happen in the framework of translation surfaces and we can understand this situation using horizon saddle connections.\newline

A \textit{chamber} is a dilation surface with boundary formed by a pentagon with two pairs of parallel sides glued on each other. The remaining side is the boundary of the chamber. Gluing the boundaries of two chambers provides a dilation surface of genus two with one singularity.\newline

Clearly, the boundary of a chamber is a $1$-horizon saddle connection. Since such a chamber cuts out the surface into two connected components, a trajectory that crosses it should cross it again in the reverse way.\newline

Considering only surfaces without horizon saddle connections, we eliminate the exceptional component formed by surfaces with two chambers. Thus, we could expect that strata of dilation surfaces without horizon saddle connections have exactly the same connected components as translation surfaces.

\subsection{Hopf and quasi-Hopf surfaces}

Hopf surfaces are examples of non-triangulable dilation surfaces. They are the only case of surfaces of genus at least two where the Veech group is not discrete (see Section 4 for details). They are also the only surfaces whose saddle connections all belong to the same direction.

\begin{defn}
A Hopf surface is a dilation surface covered by a union of disjoint affine cylinders of angle $k\pi$ where $k$ is an integer number and such that the saddle connections of the boundary of every affine cylinder lie in the same directions. In particular, every saddle connection of a Hopf surface is a $1$-horizon saddle connection.
\end{defn}

We introduce a mild generalization of Hopf surfaces.

\begin{defn}
A dilation surface $X$ is quasi-Hopf if there is a pair of directions $\alpha$ and $\beta$ such that $X$ is covered by a union of disjoint affine cylinders whose boundary saddle connections belong to directions $\alpha$ or $\beta$. We distinguish integer affine cylinders (whose angle is an integer multiple of $\pi$) from non-integer affine cylinders. 
\end{defn}

Some quasi-Hopf surfaces can be conjugated (using the action of $SL(2,\mathbb{R})$) to surfaces whose affine cylinders have an angle of $\frac{\pi}{2}$. They are examples of triangulable surfaces that admits nevertheless horizon saddle connections, see Figure 2.

\begin{figure}
\includegraphics[scale=0.3]{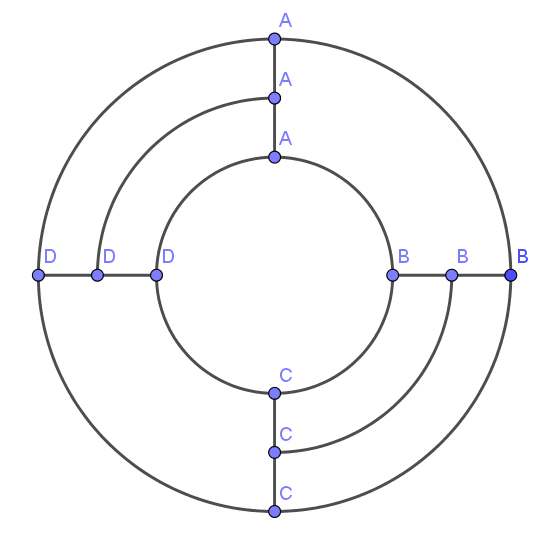}
\caption{A quasi-Hopf surface of genus $2$ with four conical singularities of angle $4\pi$. It is covered by six affine cylinders of angle $\frac{\pi}{2}$. The eight saddle connections drawn are horizon.}
\end{figure}

\begin{prop}
In a quasi-Hopf surface, if the union of the directions of the closed geodesics of any two consecutive cylinders in the decomposition is the whole circle of directions (except perhaps $\alpha$ and $\beta$), then every boundary saddle connection of a cylinder in the decomposition is a $1$-horizon saddle connection. There is no other horizon saddle connection.
\end{prop}

\begin{proof}
A horizon saddle connection cannot cross a cylinder because it would be crossed infinitely many times by a closed geodesic. Therefore, they belong to the boundary of the cylinders of the decomposition. Each of them is a $1$-horizon saddle connection because it is the boundary of two cylinders that together admits a closed geodesic in every direction (excepted $\alpha$ and $\beta$). Therefore, any trajectory crossing such a saddle connection remains forever in one of the two cylinders.
\end{proof}

The action of the one-parameter subgroup $A$ of diagonal matrices of $SL(2,\mathbb{R})$ contracting direction $\alpha$ and expanding direction $\beta$ on a quasi-Hopf surface is interesting. It crucially depends on the commensurability of the dilation ratios of affine cylinders.

\begin{prop}
We consider a quasi-Hopf surface $X$ formed by affine cylinders whose boundary saddle connections belong to exactly two direction $\alpha$ and $\beta$. The $A$-orbit of $X$ in the stratum is closed if and only the dilation ratios of non-integer affine cylinders are log-commensurable.
\end{prop}

\begin{proof}
For any quasi-Hopf surface, we consider the set $C$ of non-integer affine cylinders (we also choose an orientation on them in such a way that it goes from direction $\alpha$ to direction $\beta$). For every cylinder $i \in C$, $\lambda_{i}$ is the linear holonomy along the hyperbolic geodesics of the cylinder. We consider the flow $A^{t}$ that contracts direction $\alpha$ with a factor $e^{-t}$ and expands direction $\beta$ with a factor $e^{t}$.\newline
The dilation action preserves any affine cylinder of $C$ but modifies their twist, see Subsection 2.4. Every non-integer affine cylinder of the decomposition is preserved (with its twist) by any element $A_{t}$ of the flow such that $t=\frac{ln(\lambda)}{2}$. Besides, there could be additional symmetries such that the subgroup that preserves the cylinder and its twist is generated by $t=\theta_{i}=\frac{ln(\lambda)}{2d_{i}}$ where $d_{i}$ is an integer. These exponents $(\theta_{1},\dots,\theta_{c})$ define a \textit{characteristic ratio}. The action of $A^{t}$ on a cylinder $i$ only depends on the class of $t$ in $ \mathbb{R}/\theta_{i}\mathbb{Z}$. Equivalently, we could consider the image of $t$ in $\mathbb{R}^{C}/\mathbb{Z}^{C}$ by the map $t \mapsto (\theta_{1}t,\dots,\theta_{c}t)$. Clearly, the $A$-orbit is closed if and only if exponents $(\theta_{i})_{i \in C}$ are commensurable. In other words, exponents $(ln(\lambda_{i}))_{i\in C}$ should be commensurable.
\end{proof}

The latter condition defines the subclass of \textit{rational quasi-Hopf surfaces}.

\subsection{Dilation surfaces with finitely many saddle connections}

An open question raised in \cite{BFG} is about characterization of dilation surfaces with finitely many saddle connections. A related question asks if for any dilation surfaces, every point belongs to a saddle connection or a closed geodesic. Surfaces with finitely many saddle connections provides easy examples of this phenomenon. It also exemplifies $2$-horizon saddle connections, see Figure 3. We can generalize this example to get $k$-horizon saddle connections for an arbitrary number $k$.\newline

We provide here a characterization of dilation surfaces with finitely many saddle connections in terms of types of trajectories. Trajectories are oriented. We say that a trajectory is \textit{critical} if it starts from a conical singularity (and is not a saddle connection).\newline

\begin{figure}
\includegraphics[scale=0.3]{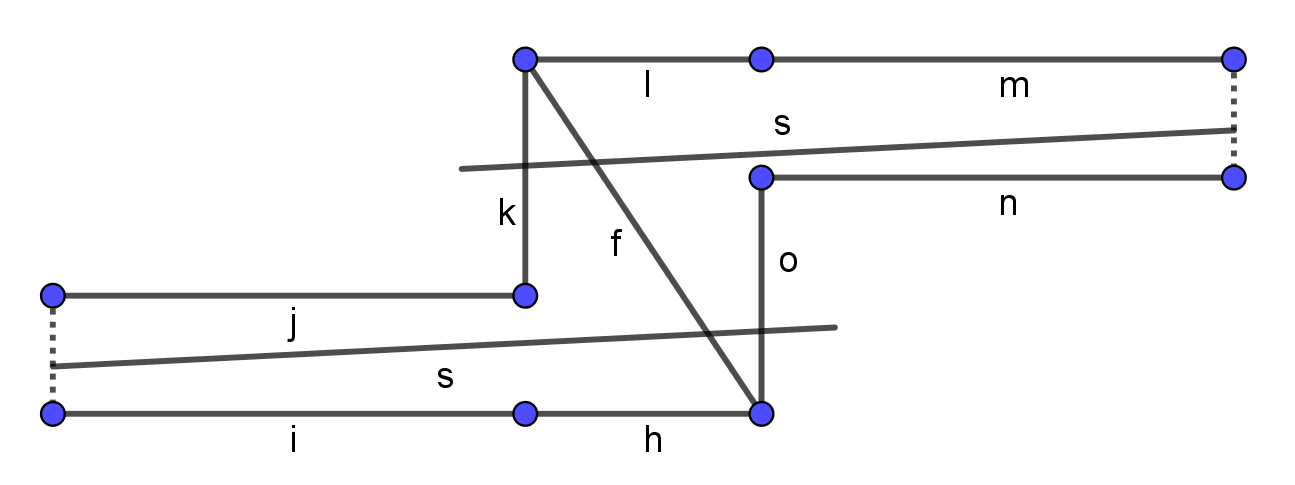}
\caption{A dilation surface where sides $h,i,j,k,l,m,n,o$ are $1$-horizon saddle connections (boundaries of affine cylinders of angle at least $\pi$) whereas saddle connection $f$ is $2$-horizon. Trajectory $s$ is an example of trajectory cutting twice $f$.}
\end{figure}

\begin{thm}
For a dilation surface $X$, the following statements are equivalent:\newline
(i) There are finitely many saddle connections in $X$.\newline
(ii) Every trajectory either hits a conical singularity or accumulates on a closed geodesic of an affine cylinder of angle at least $\pi$.
\end{thm}

\begin{proof}
We first prove that Proposition (i) implies Proposition (ii). For a surface $X$ with finitely many saddle connections, around every conical singularity, the directions of critical trajectories (finite cover of the unit circle) are divided into finitely many \textit{sectors} separated by the directions of saddle connections. The critical trajectories of a given sector form the immersion of an infinite cone in $X$ for the following reason. If the immersion of the singular sector cannot be extended in some point, then, up to zooming close to this point, we can find an embedding of the disk that does not extend to its boundary. Veech's theorem (see Appendix of \cite{DFG}) then implies that the sector in fact belongs to a hyperbolic cylinder of angle at least $\pi$ (and the angle of the sector is at least $\pi$).\newline
Thus, a small enough neighborhood of the singularity in such a cone either belongs to the triangulable locus or to a cylinder of angle at least $\pi$ (the boundary of the triangulable locus is a union of saddle connections). In the latter case, the angle of the sector is exactly $\pi$ and every critical trajectory of the sector belongs to this cylinder (and accumulates on one of its hyperbolic geodesics). The other sectors have an angle strictly smaller than $\pi$ because there are enough saddle connections to get a geometric triangulation.\newline
There is an incidence relation on the set of sectors. The left boundary of the image of the infinite cone of a sector $A$ is formed by saddle connections and at most one infinite critical trajectory. There are finitely many saddle connections so either the boundary contains an infinite critical trajectory or there is a periodicity in the saddle connections of the boundary. In the latter case, these saddle connections are the boundary of a cylinder. Since there are infinitely many saddle connections in a flat cylinder or a cylinder of angle smaller than $\pi$, this cylinder is one of those with an angle of magnitude at least $\pi$. If there is no such periodicity, the boundary contains an infinite critical trajectory that belongs to another sector $B$. The critical trajectories of sector $B$ that belong to the directions of sector $A$ are entirely contained in the infinite cone of sector $A$. This implies in particular that the total angle of sector $B$ is strictly bigger than that of sector $A$, see Figure 4. Since there are finitely many sectors, the sectors of the triangulable locus of $X$ with the biggest angle are thus incident to sectors of angle $\pi$ (those that belong to affine cylinders of angle at least $\pi$). If every critical trajectory of a sector finally enters into some affine cylinders of angle at least $\pi$, then it is clearly the same for every sector that is incident to it (the property to accumulate in some closed geodesic of a cylinder is an open condition of the direction of the trajectory). We prove thus step by step that in every sector, every critical trajectory ends in some affine cylinder of angle at least $\pi$.\newline

\begin{figure}
\includegraphics[scale=0.3]{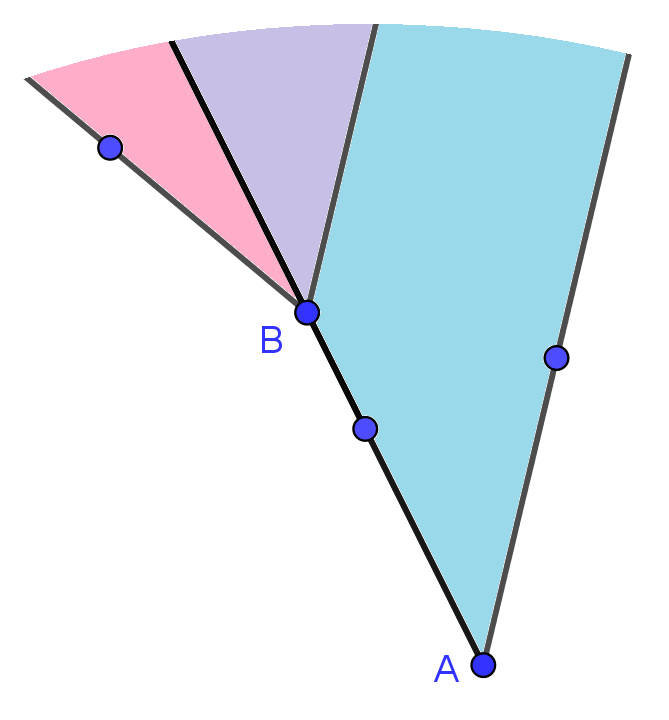}
\caption{The infinite cones defined by sectors $A$ and $B$ are in blue and red. The intersection of these two cones is in purple.}
\end{figure}

Then, for any direction $\theta$ (without loss of generality, we assume it is the horizontal direction), we draw the saddle connections and the critical trajectories in this direction. We get finitely many infinite horizontal strips (no flat cylinders since there are finitely many saddle connections). The intersection of any critical trajectory with the triangulable locus is compact. Therefore, in each of these strips, each horizontal trajectories goes from the same affine cylinder to the same affine cylinder.\newline

Then, we prove that that Proposition (ii) implies Proposition (i). Proposition (ii) implies in particular that every trajectory starting from a conical singularity either is a saddle connection or accumulates on a closed geodesic of an affine cylinder. If moreover there are infinitely many saddle connections in the surface, then there is a conical singularity $A$ starting from which a critical trajectory $\gamma$ is approached in direction by a sequence of saddle connections $(\gamma_{n})_{n \in \mathbb{N}}$ starting from $A$. Without loss of generality, we assume the sequence is monotonically converging in the counterclockwise direction. If $\gamma$ finally enters into an affine cylinder of angle at least $\pi$ and accumulates on a closed geodesic, then this is the same for every critical trajectory starting from $A$ with a direction in some open neighborhood of $\theta$. This contradicts accumulation of saddle connections in this direction. On the contrary, if $\gamma$ is a saddle connection from $A$ to another singularity $B$, then sequence $(\gamma_{n})_{n \in \mathbb{N}}$ also accumulates on the trajectory $\delta$ starting from $B$ in the direction obtained by a counterclockwise rotation of angle $\pi$ from the ending direction of $\gamma$. The infinite sequence of saddle connections just passes $\gamma$ and $B$ from the right, see Figure 5. The same reasoning holds for $\delta$ if it is also a saddle connection. Finally, we automatically get a contradiction with the existence of an accumulation direction of saddle connections. Consequently, there are finitely many of them.

\begin{figure}
\includegraphics[scale=0.3]{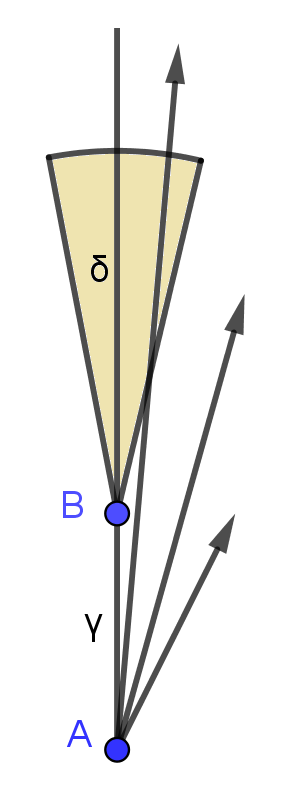}
\caption{Saddle connection $\gamma$ continued by critical trajectory $\delta$ on which accumulates a sequence of saddle connections. The infinite open cone starting from sector $B$ formed by trajectories entering in some affine cylinder is in yellow.}
\end{figure}
\end{proof}

\section{Veech groups}

\subsection{General results}

The first structure theorem for Veech groups of dilation surfaces has been proved in \cite{DFG}.

\begin{thm}[Dichotomy theorem]
Let $X$ be a dilation surface of genus $g \geq 2$, then there are two possible cases:\newline
(i) $X$ is a Hopf surface and $V(X)$ is conjugated to the subgroup of upper triangular elements of $SL(2,\mathbb{R})$;\newline
(ii) $V(X)$ is a discrete subgroup of $SL(2,\mathbb{R})$.
\end{thm}

The condition on the genus is necessary because there are some affine tori whose Veech group is $SL(2,\mathbb{R})$. They are constructed as the exponential of some flat tori (see Subsection 4.4 of \cite{DFG} for details).\newline
\newline
Parabolic directions in translations surfaces correspond to cylinder decompositions. There is a similar result for dilation surfaces.

\begin{prop}
Let $X$ be a dilation surface such that $V(X)$ contains a parabolic element. Then the decomposition of $X$ into invariant components of the parabolic direction is formed by affine and flat cylinders whose boundary saddle connections belong to the parabolic direction. Moduli of the flat cylinders should be commensurable.
\end{prop}

\begin{proof}
We first follow a part of the proof of Proposition 4 in \cite{DFG} to prove that in parabolic directions, critical trajectories are saddle connections. By contradiction, we assume there is an infinite critical trajectory $\gamma$ in a parabolic direction. Without loss of generality, we assume $\gamma$ is vertical. Since it is infinite, $\gamma$ has an accumulation point $x$. Any neighborhood $U$ of $x$ is crossed infinitely many times by $\gamma$. Accumulation point $x$ is approached in $U$ by an infinite sequence of vertical lines. Any element of $GL(2,\mathbb{R})$ that acts as a scaling on $\gamma$ also acts as a scaling on $U \cap \gamma$. Consequently, it acts as the same scaling on $U$. Therefore, any element of $GL(2,\mathbb{R})$ that preserves the vertical direction in this dilation surface is the identity up to a scalar factor. It is thus the identity in the Veech group of the dilation structure (elements are considered up to a scalar factor since lengths are not globally preserved). We proved that there cannot be parabolic elements that preserve the direction of an infinite critical trajectory. Every critical trajectory in a parabolic direction is a saddle connection.\newline
Cutting along saddle connections in the parabolic direction provides a decomposition of the dilation surface into invariants components. Since conical singularities are cut out into sectors of angle equal to $\pi$, discrete Gauss-Bonnet implies that the components are topological annuli. These components are thus either flat cylinders or affine cylinders the angle of which are integer multiples of $\pi$. Since the flat cylinders are preserved by the parabolic element, their moduli should be commensurable. 
\end{proof}

Hyperbolic directions in a dilation surface are directions that are preserved by an hyperbolic element of the Veech group of the surface. There is no closed geodesics nor saddle connections in hyperbolic directions of translation surfaces. Such phenomena can appear for some specific dilation surfaces.

\begin{lem}
Let $X$ be a dilation surface such that $V(X)$ contains a hyperbolic element $\phi$. If a saddle connection $\gamma$ lies in a hyperbolic direction, then $\gamma$ belongs to the common boundary of two cylinders (see Figure 3 for examples of cylinders with only a part of the boundary in common).
\end{lem}

\begin{proof}
Without loss of generality, we can assume $\gamma$ is vertical and that $\phi$ is expanding in the vertical direction and contracting in the horizontal direction. If $\gamma$ is an edge of some triangle $A$ (whose other sides are not in the vertical direction), then $\phi$ provides a sequence of triangles bounding $\gamma$ and such that the directions of the other edges approach the vertical direction. There cannot be a horizontal side $\alpha$ in triangle $A$e because in this case the action of $\phi$ on the triangle would provide another triangle $B$ with $\gamma$ as a vertical side and a portion of $\alpha$ as a horizontal side. This is impossible because there is no singularity in the interior of a saddle connection.\newline
If none of the sides is horizontal, their local length ratio (see Subsection 2.1 for details) is bounded by below by the ratios of their vertical coordinates. Therefore, either there is a conical singularity inside $\gamma$ (which is impossible) or there is another saddle connection that forms an angle of $\pi$ with $\gamma$. We can iterate this procedure to get a chain of saddle connections in the vertical direction. Since there is a finite number of them, these saddle connections are the boundary of a cylinder. If $\gamma$ is not an edge of a triangle, then it is a boundary saddle connection of an affine cylinder of angle at least $\pi$. 
\end{proof}

\begin{prop}
Let $X$ be a dilation surface such that $V(X)$ contains a hyperbolic element $\phi$. If a saddle connection or a closed geodesic belong to a hyperbolic direction, then $X$ is a quasi-Hopf surface.
\end{prop}

\begin{proof}
If an affine cylinder is preserved by an hyperbolic element of the Veech group, then its boundary saddle connections belong to invariant directions of the element. Besides, there is no flat cylinder whose closed geodesics belong to an hyperbolic direction. In a given hyperbolic direction, there are at most finitely many saddle connections or closed geodesics. Therefore, at least one power of the hyperbolic element preserves the cylinder they belong to. Consequently, we get invariant cylinders whose boundary is connected with other invariant cylinders until all the surface is covered.
\end{proof}

\subsection{Veech groups in presence of horizon saddle connections}

There is a lot of examples of dilation surfaces with many directions of horizon saddle connections. Following Proposition 3.4, some dilation tori feature cylinder decompositions such that every boundary saddle connections is horizon. We can also glue chambers (see Subsection 3.2) on the boundary of a polygon. If there are at least three distinct directions of horizon saddle connections, then the Veech group of the surface is finite.

\begin{proof}[Proof of Theorem 1.3]
The set of directions of horizon saddle connections is finite and globally preserved by elements of the Veech group. Therefore, for any element $\phi$ of the Veech group, there is an integer $n$ such that $\phi^{n}$ preserves every direction of horizon saddle connection. Since they are at least three, then $\phi^{n}$ is the identity. Thus, any element of the Veech group is elliptic. Following Theorem 1 in \cite{DFG}, the Veech group of a dilation surface that is not a Hopf surface is discrete. Therefore, the Veech group of $X$ is conjugated to a finite rotation group.
\end{proof}

In the case of dilation surfaces with exactly two directions of horizon saddle connections, we have to take into account the specific situation of quasi-Hopf surfaces we introduced previously.

\begin{proof}[Proof of Theorem 1.4]
A parabolic element cannot globally preserve two distinct directions so the Veech group is discrete (see Theorem 1 of \cite{DFG}) and contains only elliptic and hyperbolic elements. Elliptic elements globally preserve the two distinct directions so they can only preserve each of them or intertwine them. Therefore, elliptic elements form a finite group conjugated to the trivial group or the finite group of rotations of order two or four.\newline
If the Veech group contains a hyperbolic element, then this element preserves the two directions of horizon saddle connections. Every hyperbolic element preserves the same pair of directions. Since the surface is not Hopf (otherwise there would be only one direction of horizon saddle connections), the Veech group is discrete so the hyperbolic elements belong to the same cyclic group. Following Proposition 4.4, the surface is then quasi-Hopf. Proposition 3.8 provides the condition a quasi-Hopf surfaces should satisfy to admis a hyperbolic element of the Veech group that preserves the two directions of the boundary saddle connections of the cylinders.
\end{proof}

The case of dilation surfaces with only one direction of horizon saddle connections includes that of Hopf surfaces.

\begin{proof}[Proof of Theorem 1.5]
The elliptic elements preserve the direction of horizon saddle connections so they belong to the group $\left\{ \pm Id \right\}$. Every parabolic element preserves this direction. Therefore, either the Veech group is not discrete (the surface is then a Hopf surface and its described by Theorem 4.1) or the parabolic elements all belong to the same cyclic parabolic group. If there is an hyperbolic element in the Veech group of the surface, then it preserves the horizon saddle connections. These saddle connections thus are boundaries of cylinders whose boundary saddle connections belong to one of the two hyperbolic directions. The surface is then quasi-Hopf. If there is only one direction of horizon saddle connections, then the surface is Hopf.
\end{proof}

\textit{Acknowledgements.} The author is supported by a fellowship of Weizmann Institute of Science. This research was supported by the Israel Science Foundation (grant No.  1167/17). The author thanks Selim Ghazouani for interesting discussions. The author also thanks the anonymous referee for their patience and their valuable remarks.\newline

\nopagebreak
\vskip.5cm
\end{document}